\documentclass[a4paper,english,fontsize=11pt,parskip=half,abstracton]{scrartcl}
\usepackage{babel}
\usepackage[utf8]{inputenc}
\usepackage[T1]{fontenc}
\usepackage[a4paper,left=20mm,right=20mm,top=30mm,bottom=30mm]{geometry}
\usepackage{amsmath}
\usepackage{amsthm}
\usepackage{amssymb}
\usepackage{enumerate}
\usepackage{aliascnt}
\usepackage[bookmarks=true,
            pdftitle={Landau's Theorem for pi-blocks of pi-separable groups},
            pdfauthor={Benjamin Sambale},
            pdfkeywords={Brauer's Problem 21,pi-blocks,number of characters},
            pdfstartview={FitH}]{hyperref}

\newtheorem{Thm}{Theorem} 
\newtheorem*{ThmA}{Theorem~A}
\newtheorem*{ThmB}{Theorem~B}
\newaliascnt{Lem}{Thm}
\newtheorem{Lem}[Lem]{Lemma}
\aliascntresetthe{Lem}
\newaliascnt{Prop}{Thm}
\newtheorem{Prop}[Prop]{Proposition}
\aliascntresetthe{Prop}
\newaliascnt{Cor}{Thm}

\aliascntresetthe{Cor}
\newaliascnt{Con}{Thm}

\aliascntresetthe{Con}
\theoremstyle{definition}
\newaliascnt{Def}{Thm}
\newtheorem{Def}[Def]{Definition}
\aliascntresetthe{Def}
\newaliascnt{Ex}{Thm}

\aliascntresetthe{Ex}

\numberwithin{equation}{section}

\allowdisplaybreaks[1]

\renewcommand{\phi}{\varphi}
\newcommand{\C}{\operatorname{C}}

\newcommand{\Z}{\operatorname{Z}}

\newcommand{\NN}{\mathbb{N}}

\newcommand{\Aut}{\operatorname{Aut}}

\newcommand{\pcore}{\operatorname{O}}

\newcommand{\Irr}{\operatorname{Irr}}

\newcommand{\Ker}{\operatorname{Ker}}

\mathchardef\ordinarycolon\mathcode`\:  
 \mathcode`\:=\string"8000
 \begingroup \catcode`\:=\active
   \gdef:{\mathrel{\mathop\ordinarycolon}}
 \endgroup

\title{Landau's Theorem for $\pi$-blocks\\ of $\pi$-separable groups}
\author{Benjamin Sambale\footnote{Institut für Mathematik, Friedrich-Schiller-Universität Jena, 07737 Jena, Germany, 
\href{mailto:benjamin.sambale@uni-jena.de}{benjamin.sambale@uni-jena.de}}}
\date{\today}

\begin{document}
\frenchspacing
\maketitle
\begin{abstract}\noindent
Slattery has generalized Brauer's theory of $p$-blocks of finite groups to $\pi$-blocks of $\pi$-separable groups where $\pi$ is a set of primes. In this setting we show that the order of a defect group of a $\pi$-block $B$ is bounded in terms of the number of irreducible characters in $B$. This is a variant of Brauer's Problem 21 and generalizes Külshammer's corresponding theorem for $p$-blocks of $p$-solvable groups. At the same time, our result generalizes Landau's classical theorem on the number of conjugacy classes of an arbitrary finite group. The proof relies on the classification of finite simple groups.
\end{abstract}


\section{Introduction}

Many authors, including Richard Brauer himself, have tried to replace the prime $p$ in modular representation theory by a set of primes $\pi$. One of the most convincing settings is the theory of $\pi$-blocks of $\pi$-separable groups which was developed by Slattery~\cite{Slattery,Slattery2} building on the work of Isaacs and others (for precise definitions see next section). In this framework most of the classical theorems on $p$-blocks can be carried over to $\pi$-blocks. For instance, Slattery proved versions of Brauer's three main theorems for $\pi$-blocks. Also many of the open conjectures on $p$-blocks make sense for $\pi$-blocks. 
In particular, \emph{Brauer's Height Zero Conjecture} and the \emph{Alperin--McKay Conjecture} for $\pi$-blocks of $\pi$-separable groups were verified by Manz--Staszewski~\cite[Theorem~3.3]{ManzStaszewski} and Wolf~\cite[Theorem~2.2]{Wolf} respectively. In a previous paper~\cite{SambalePi} the present author proved \emph{Brauer's $k(B)$-Conjecture} for $\pi$-blocks of $\pi$-separable groups. This means that the number $k(B)$ of irreducible characters in a $\pi$-block $B$ is bounded by the order of its defect groups. 

In this paper we work in the opposite direction. Landau's classical theorem asserts that the order of a finite group $G$ can be bounded by a function depending only on the number of conjugacy classes of $G$.
\emph{Problem~21} on Brauer's famous list~\cite{BrauerLectures} from 1963 asks if the order of a defect group of a block $B$ of a finite group can be bounded by a function depending only on $k(B)$. 
Even today we do not know if there is such a bound for blocks with just three irreducible characters (it is expected that the defect groups have order three in this case, see \cite[Chapter~15]{habil}). On the other hand, an affirmative answer to Problem~21 for $p$-blocks of $p$-solvable groups was given by Külshammer~\cite{KLandau3}. Moreover, Külshammer--Robinson~\cite{KR} showed that a positive answer in general would follow from the Alperin--McKay Conjecture. 

The main theorem of this paper settles Problem~21 for $\pi$-blocks of $\pi$-separable groups. 

\begin{ThmA}
The order of a defect group of a $\pi$-block $B$ of a $\pi$-separable group can be bounded by a function depending only on $k(B)$.
\end{ThmA}

Since $\{p\}$-separable groups are $p$-solvable and $\{p\}$-blocks are $p$-blocks, this generalizes Külshammer's result.
If $G$ is an arbitrary finite group and $\pi$ is the set of prime divisors of $|G|$, then $G$ is $\pi$-separable and $\Irr(G)$ is a $\pi$-block with defect group $G$ (see \autoref{facts} below). Hence, Theorem~A also implies Landau's Theorem mentioned above.

Külshammer's proof relies on the classification of finite simple groups and so does our proof.
Although it is possible to extract from the proof an explicit bound on the order of a defect group, this bound is far from being optimal. With some effort we obtain the following small values.

\begin{ThmB}
Let $B$ be a $\pi$-block of a $\pi$-separable group with defect group $D$. Then
\begin{align*}
k(B)=1&\Longleftrightarrow D=1,\\
k(B)=2&\Longleftrightarrow D=C_2,\\
k(B)=3&\Longleftrightarrow D\in\{C_3,S_3\}
\end{align*}
where $C_n$ denotes the cyclic group of order $n$ and $S_n$ is the symmetric group of degree $n$.
\end{ThmB}

\section{Notation}

Most of our notation is standard and can be found in Navarro's book~\cite{Navarro}.
For the convenience of the reader we collect definitions and crucial facts about $\pi$-blocks. 
In the following, $\pi$ is any set of prime numbers. We denote the $\pi$-part of an integer $n$ by $n_\pi$.
A finite group $G$ is called $\pi$-\emph{separable} if there exists a normal series 
\[1=N_0\unlhd\ldots\unlhd N_k=G\] 
such that each quotient $N_i/N_{i-1}$ is a $\pi$-group or a $\pi'$-group. The largest normal $\pi$-subgroup of $G$ is denoted by $\pcore_{\pi}(G)$.

\begin{Def}\hfill
\begin{itemize}
\item A $\pi$-\emph{block} of $G$ is a minimal non-empty subset $B\subseteq\Irr(G)$ such that $B$ is a union of $p$-blocks for every $p\in\pi$ (see \cite[Definition (1.12) and Theorem (2.15)]{Slattery}). In particular, the $\{p\}$-blocks of $G$ are the $p$-blocks of $G$. In accordance with the notation for $p$-blocks we set $k(B):=|B|$ for every $\pi$-block $B$.

\item A \emph{defect group} $D$ of a $\pi$-block $B$ of a $\pi$-separable group $G$ is defined inductively as follows (see \cite[Definition (2.2)]{Slattery2}). Let $\chi\in B$ and let $\lambda\in\Irr(\pcore_{\pi'}(G))$ be a constituent of the restriction $\chi_{\pcore_{\pi'}(G)}$ (we say that $B$ \emph{lies over} $\lambda$). Let $G_\lambda$ be the inertial group of $\lambda$ in $G$. If $G_\lambda=G$, then $D$ is a Hall $\pi$-subgroup of $G$ (such subgroups always exist in $\pi$-separable groups). 
Otherwise there exists a unique $\pi$-block $b$ of $G_\lambda$ lying over $\lambda$ such that $\psi^G\in B$ for any $\psi\in b$ (see \autoref{FR} below). In this case we identify $D$ with a defect group of $b$. 
As usual, the defect groups of $B$ form a conjugacy classes of $G$.
It was shown in \cite[Theorem (2.1)]{Slattery2} that this definition agrees with the usual definition for $p$-blocks.

\item A $\pi$-block $B$ of $G$ \emph{covers} a $\pi$-block $b$ of $N\unlhd G$, if there exist $\chi\in B$ and $\psi\in b$ such that $[\chi_N,\psi]\ne 0$ (see \cite[Definition (2.5)]{Slattery}). 
\end{itemize}
\end{Def}

\begin{Prop}\label{facts}
For every $\pi$-block $B$ of a $\pi$-separable group $G$ with defect group $D$ the following holds:
\begin{enumerate}[(i)]
\item\label{f1} $\pcore_{\pi}(G)\le D$.
\item\label{f2} For every $\chi\in B$ we have $\frac{|D|\chi(1)_\pi}{|G|_\pi}\in\NN$ and for some $\chi$ this fraction equals $1$.
\item\label{f3} If the $\pi$-elements $g,h\in G$ are not conjugate, then
\[\sum_{\chi\in B}\chi(g)\overline{\chi(h)}=0.\]
\item\label{f4} If $B$ covers a $\pi$-block $b$ of $N\unlhd G$, then for every $\psi\in b$ there exists some $\chi\in B$ such that $[\chi_N,\psi]\ne 0$.
\item\label{f5} If $B$ lies over a $G$-invariant $\lambda\in\Irr(\pcore_{\pi'}(G))$, then $B=\Irr(G|\lambda)$.
\end{enumerate}
\end{Prop}
\begin{proof}\hfill
\begin{enumerate}[(i)]
\item See \cite[Lemma~(2.3)]{Slattery2}.
\item See \cite[Theorems (2.5) and (2.15)]{Slattery2}.
\item This follows from \cite[Corollary~8]{Robinsonpi} (by \cite[Remarks on p. 410]{Robinsonpi}, $B$ is really a $\pi$-block in the sense of that paper). 
\item See \cite[Lemma (2.4)]{Slattery}.
\item See \cite[Theorem~(2.8)]{Slattery}.\qedhere
\end{enumerate}
\end{proof}

The following result allows inductive arguments (see \cite[Theorem~2.10]{Slattery} and \cite[Corollary~2.8]{Slattery2}).

\begin{Thm}[Fong--Reynolds Theorem for $\pi$-blocks]\label{FR}
Let $N$ be a normal $\pi'$-subgroup of a $\pi$-separable group $G$. Let $\lambda\in\Irr(N)$ with inertial group $G_\lambda$. Then the induction of characters induces a bijection $b\mapsto b^G$ between the $\pi$-blocks of $G_\lambda$ lying over $\lambda$ and the $\pi$-blocks of $G$ lying over $\lambda$. Moreover, $k(b)=k(b^G)$ and every defect group of $b$ is a defect group of $b^G$.
\end{Thm}

Finally we recall $\pi$-special characters which were introduced by Gajendragadkar~\cite{Gajendragadkar}. A character $\chi\in\Irr(G)$ is called $\pi$-\emph{special}, if $\chi(1)=\chi(1)_\pi$ and for every subnormal subgroup $N$ of $G$ and every irreducible constituent $\phi$ of $\chi_N$ the order of the linear character $\det\phi$ is a $\pi$-number. 

Obviously, every character of a $\pi$-group is $\pi$-special.
If $\chi\in\Irr(G)$ is $\pi$-special and $N:=\pcore_{\pi'}(G)$, then $\chi_N$ is a sum of $G$-conjugates of a linear character $\lambda\in\Irr(N)$ by Clifford theory. Since the order of $\det\lambda=\lambda$ is a $\pi$-number and divides $|N|$, we obtain $\lambda=1_N$. This shows that $N\le\Ker(\chi)$.

\section{Proofs}

At some point in the proof of Theorem~A we need to refer to Külshammer's solution~\cite[Theorem]{KLandau3} of Brauer's Problem 21 for $p$-solvable groups:

\begin{Prop}\label{K}
There exists a monotonic function $\alpha:\NN\to\NN$ with the following property: For every $p$-block $B$ of a $p$-solvable group with defect group $D$ we have $|D|\le \alpha(k(B))$. 
\end{Prop}

The following ingredient is a direct consequence of the classification of finite simple groups.

\begin{Prop}[{\cite[Theorem~2.1]{Kohl}}]\label{kohl}
There exists a monotonic function $\beta:\NN\to\NN$ with the following property: If $G$ is a finite non-abelian simple group such that $\Aut(G)$ has exactly $k$ orbits on $G$, then $|G|\le \beta(k)$.
\end{Prop}

In the following series of lemmas, $\pi$ is a fixed set of primes, $G$ is a $\pi$-separable group and $B$ is a $\pi$-block of $G$ with defect group $D$. \autoref{facts}\eqref{f2} guarantees the existence of a height $0$ character in $B$. We need to impose an additional condition on such a character.

\begin{Lem}\label{lemchar}
There exists some $\chi\in B$ such that $\pcore_{\pi}(G)\le\Ker(\chi)$ and $|D|\chi(1)_\pi=|G|_\pi$.
\end{Lem}
\begin{proof}
We argue by induction on $|G|$. Let $B$ lie over $\lambda\in\Irr(\pcore_{\pi'}(G))$. Suppose first that $G_\lambda=G$. Then $|D|=|G|_\pi$ and $B=\Irr(G|\lambda)$ by \autoref{facts}\eqref{f5}.
Since $\lambda$ is $\pi'$-special, there exists a $\pi'$-special $\chi\in B$ by \cite[Lemma~(2.7)]{Slattery}. It follows that $\pcore_\pi(G)\le\Ker(\chi)$ and $|D|\chi(1)_\pi=|D|=|G|_\pi$.

Now assume that $G_\lambda<G$. Let $b$ be the Fong--Reynolds correspondent of $B$ in $G_\lambda$. By induction there exists some $\psi\in b$ such that $\pcore_\pi(G_\lambda)\le\Ker(\psi)$ and $|D|\psi(1)_\pi=|G_\lambda|_\pi$. Let $\chi:=\psi^G\in B$. Since $[\pcore_\pi(G),\pcore_{\pi'}(G)]=1$ we have $\pcore_\pi(G)\le\pcore_\pi(G_\lambda)\le\Ker(\psi)$ and $\pcore_\pi(G)\le\Ker(\chi)$ (see \cite[Lemma~(5.11)]{Isaacs}). Finally, \[|D|\chi(1)_\pi=|D|\psi(1)_\pi|G:G_\lambda|_\pi=|G_\lambda|_\pi|G:G_\lambda|_\pi=|G|_\pi.\qedhere\]
\end{proof}

For every $p\in\pi$, the character $\chi$ in \autoref{lemchar} lies in a $p$-block $B_p\subseteq B$ whose defect group has order $|D|_p$. In fact, it is easy to show that every Sylow $p$-subgroup of $D$ is a defect group of $B_p$.

Our second lemma extends an elementary fact on $p$-blocks (see \cite[Theorem (9.9)(b)]{Navarro}).

\begin{Lem}\label{lemquot}
Let $N$ be a normal $\pi$-subgroup of $G$. Then $B$ contains a $\pi$-block of $G/N$ with defect group $D/N$. 
\end{Lem}
\begin{proof}
Again we argue by induction on $|G|$. Let $\lambda\in\Irr(\pcore_{\pi'}(G))$ be under $B$. Suppose first that $G_\lambda=G$. Then $|D|=|G|_\pi$. By \autoref{lemchar}, there exists some $\chi\in B$ such that $N\le\pcore_\pi(G)\le\Ker(\chi)$ and $\chi(1)_\pi=1$. Hence, we may consider $\chi$ as a character of $\overline{G}:=G/N$. As such, $\chi$ lies in a $\pi$-block $\overline{B}$ of $\overline{G}$.
For any $\psi\in\overline{B}$ there exists a sequence of characters $\chi=\chi_1,\ldots,\chi_k=\psi$ such that $\chi_i$ and $\chi_{i+1}$ lie in the same $p$-block of $\overline{G}$ for some $p\in\pi$ and $i=1,\ldots,k-1$. Then $\chi_i$ and $\chi_{i+1}$ also lie in the same $p$-block of $G$. This shows that $\psi\in B$ and $\overline{B}\subseteq B$. For a defect group $P/N$ of $\overline{B}$ we have
\[|P/N|=\max\Bigl\{\frac{|\overline{G}|_\pi}{\psi(1)_\pi}:\psi\in\overline{B}\Bigr\}=\frac{|\overline{G}|_\pi}{\chi(1)_\pi}=|\overline{G}|_\pi=|D/N|\]
by \autoref{facts}\eqref{f2}. Since the Hall $\pi$-subgroups are conjugate in $G$, we conclude that $D/N$ is a defect group of $\overline{B}$.

Now let $G_\lambda<G$, and let $b$ be the Fong--Reynolds correspondent of $B$ in $G_\lambda$. After conjugation, we may assume that $D$ is a defect group of $b$. By induction, $b$ contains a block $\overline{b}$ of $G_\lambda/N$ with defect group $D/N$. If we regard $\lambda$ as a character of $\pcore_{\pi'}(G)N/N\cong\pcore_{\pi'}(G)$, we see that $\overline{G}_\lambda=G_\lambda/N$. It follows that the Fong--Reynolds correspondent $\overline{B}=\overline{b}^{\overline{G}}$ of $\overline{b}$ is contained in $B$ and has defect group $D/N$.
\end{proof}

The next result extends one half of \cite[Proposition]{KLandau1} to $\pi$-blocks.

\begin{Lem}\label{lemsub}
Let $N\unlhd G$, and let $b$ be a $\pi$-block of $N$ covered by $B$. Then $k(b)\le|G:N|k(B)$.
\end{Lem}
\begin{proof}
By \autoref{facts}\eqref{f4}, $b\subseteq\bigcup_{\chi\in B}{\Irr(N|\chi)}$. For every $\chi\in B$ the restriction $\chi_N$ is a sum of $G$-conjugate characters according to Clifford theory. In particular, $\lvert\Irr(N|\chi)\rvert\le|G:N|$ and
\[k(b)\le\sum_{\chi\in B}\lvert\Irr(N|\chi)\rvert\le |G:N|k(B).\qedhere\]
\end{proof}

It is well-known that the number of irreducible characters in a $p$-block $B$ is greater or equal than the number of conjugacy classes which intersect a given defect group of $B$ (see \cite[Problem (5.7)]{Navarro}). For $\pi$-blocks we require the following weaker statement.

\begin{Lem}\label{lemnormal}
Let $N$ be a normal $\pi$-subgroup of $G$. Then the number of $G$-conjugacy classes contained in $N$ is at most $k(B)$.
\end{Lem}
\begin{proof}
Let $R\subseteq N$ be a set of representatives for the $G$-conjugacy classes inside $N$. By \autoref{lemchar}, there exists some $\chi\in B$ such that $\chi(r)=\chi(1)\ne 0$ for every $r\in R$. Thus, the columns of the matrix $M:=(\chi(r):\chi\in B,\,r\in R)$ are non-zero. By \autoref{facts}\eqref{f3}, the columns of $M$ are pairwise orthogonal, so in particular they are linearly independent. Hence, the number of rows of $M$ is at least $|R|$. 
\end{proof}

We can prove the main theorem now.

\begin{proof}[Proof of Theorem~A]
The proof strategy follows closely the arguments in \cite{KLandau3}.
We construct inductively a monotonic function $\gamma:\NN\to\NN$ with the desired property. To this end, let $B$ be a $\pi$-block of a $\pi$-separable group $G$ with defect group $D$ and $k:=k(B)$. If $k=1$, then the unique character in $B$ has $p$-defect $0$ for every $p\in\pi$. It follows from \autoref{facts}\eqref{f2} that this can only happen if $D=1$. Hence, let $\gamma(1):=1$. 

Now suppose that $k>1$ and $\gamma(l)$ is already defined for $l<k$. Let $N:=\pcore_{\pi'}(G)$. By a repeated application of the Fong--Reynolds Theorem for $\pi$-blocks and \autoref{facts}\eqref{f5}, we may assume that $B$ is the set of characters lying over a $G$-invariant $\lambda\in\Irr(N)$. Then $D$ is a Hall $\pi$-subgroup of $G$. By \cite[Problem~(6.3)]{Navarro2}, there exists a character triple isomorphism \[(G,N,\lambda)\to(\widehat{G},\widehat{N},\widehat{\lambda})\] 
such that $G/N\cong\widehat{G}/\widehat{N}$ and $\widehat{N}=\pcore_{\pi'}(\widehat{G})\le\Z(\widehat{G})$. Then $\widehat{B}:=\Irr(\widehat{G}|\widehat{\lambda})$ is a $\pi$-block of $\widehat{G}$ with defect group $\widehat{D}\cong D$ and $k(\widehat{B})=k$. After replacing $G$ by $\widehat{G}$ we may assume that $N\le\Z(G)$.
Then 
\[\pcore_{\pi'\pi}(G)=N\times P\] 
where $P:=\pcore_{\pi}(G)$. If $P=1$, then $G$ is a $\pi'$-group and we derive the contradiction $k=1$. Hence, $P\ne 1$.

Let $M$ be a minimal normal subgroup of $G$ contained in $P$. 
By \autoref{lemquot}, $B$ contains a $\pi$-block $\overline{B}$ of $G/M$ with defect group $D/M$. Since the kernel of $B$ is a $\pi'$-group (see \cite[Theorem~(6.10)]{Navarro}), we have $k(\overline{B})<k$. By induction, it follows that 
\begin{equation}\label{DM}
|D/M|\le \gamma(k-1)
\end{equation}
where we use that $\gamma$ is monotonic. 
Let $H/M$ be a Hall $\pi'$-subgroup of $G/M$, and let 
\[K:=\bigcap_{g\in G}gHg^{-1}\unlhd G.\] 
Then 
\[|G:K|\le|G:H|!=|G/M:H/M|!=(|G/M|_\pi)!=|D/M|!\le \gamma(k-1)!\]
by \eqref{DM}. Let $b$ be a $\pi$-block of $K$ covered by $B$. By \autoref{lemsub}, 
\begin{equation}\label{kb}
k(b)\le|G:K|k\le \gamma(k-1)!k.
\end{equation}
Thus we have reduced our problem to the block $b$ of $K$. Since $K/M\le H/M$ is a $\pi'$-group, $b$ has defect group $M$ by \autoref{facts}\eqref{f1}.

As a minimal normal subgroup, $M$ is a direct product of isomorphic simple groups. 
Suppose first that $M$ is an elementary abelian $p$-group for some $p\in\pi$. 
Then $K$ is $p$-solvable and $b$ is just a $p$-block with defect group $M$. Hence, with the notation from \autoref{K} we have 
\begin{equation}\label{M1}
|M|\le \alpha(k(b))\le \alpha\bigl(\gamma(k-1)!k\bigr)
\end{equation}
by \eqref{kb}. 

Now suppose that $M=S\times\ldots\times S=S^n$ where $S$ is a non-abelian simple group. 
Let $x_1,\ldots,x_s\in S$ be representatives for the orbits of $\Aut(S)$ on $S\setminus\{1\}$. Since $\Aut(M)\cong\Aut(S)\wr S_n$ (where $S_n$ denotes the symmetric group of degree $n$), the elements $(x_i,1,\ldots,1)$, $(x_i,x_i,1,\ldots,1),\ldots,(x_i,\ldots,x_i)$ of $M$ with $i=1,\ldots,s$ lie in distinct conjugacy classes of $K$. Consequently, \autoref{lemnormal} yields $ns\le k(b)$. Now with the notation of \autoref{kohl} we deduce that $|S|\le \beta(s+1)$ and
\begin{equation}\label{M2}
|M|=|S|^n\le \beta(s+1)^n\le \beta\bigl(k(b)+1\bigr)^{k(b)}\le \beta\bigl(\gamma(k-1)!k+1\bigr)^{\gamma(k-1)!k}
\end{equation}
by \eqref{kb}.

Setting 
\[\gamma(k):=\gamma(k-1)\max\bigl\{\alpha\bigl(\gamma(k-1)!k\bigr),\,\beta\bigl(\gamma(k-1)!k+1\bigr)^{\gamma(k-1)!k}\bigr\}\]
we obtain
\[|D|=|D/M||M|\le\gamma(k)\]
by \eqref{DM}, \eqref{M1} and \eqref{M2}.
Obviously, $\gamma$ is monotonic.
\end{proof}

\begin{proof}[Proof of Theorem~B]
We have seen in the proof of Theorem~A that $k(B)=1$ implies $D=1$. Conversely, \cite[Theorem~3]{SambalePi} shows that $D=1$ implies $k(B)=1$. 

Now let $k(B)=2$. Then $B$ is a $p$-block for some $p\in\pi$. By a result of Brandt~\cite[Theorem~A]{Brandt}, $p=2$ and $|D|_2=2$ follows from \autoref{facts}\eqref{f2}. For every $q\in\pi\setminus\{2\}$, $B$ consists of two $q$-defect $0$ characters. This implies $D=C_2$. Conversely, if $D=C_2$, then we obtain $k(B)=2$ by \cite[Theorem~3]{SambalePi}. 

Finally, assume that $k(B)=3$. As in the proof of Theorem~A, we may assume that 
\[\pcore_{\pi'\pi}(G)=\pcore_{\pi'}(G)\times P\] 
with $P:=\pcore_{\pi}(G)\ne 1$. 
By the remark after \autoref{lemchar}, for every $p\in\pi$ there exists a $p$-block contained in $B$ whose defect group has order $|D|_p$. If $|D|_2\ge 4$, we derive the contradiction $k(B)\ge 4$ by \cite[Proposition~1.31]{habil}. Hence, $|D|_2\le 2$. By \autoref{lemquot}, $B$ contains a $\pi$-block $\overline{B}$ of $G/P$ with defect group $D/P$ and $k(\overline{B})<k(B)$. The first part of the proof yields $|D/P|\le 2$. In particular, $P$ is a Hall subgroup of $D$.
From \autoref{lemnormal} we see that $P$ has at most three orbits under $\Aut(P)$.
If $P$ is an elementary abelian $p$-group, then $B$ contains a $p$-block $B_p$ with normal defect group $P$. The case $p=2$ is excluded by the second paragraph of the proof. Hence, $p>2$ and $k(B_p)=k(B)=3$. Now \cite[Proposition~15.2]{habil} implies $|P|=3$ and $|D|\in\{3,6\}$. A well-known lemma by Hall--Higman states that $\C_G(P)\le \pcore_{\pi'\pi}(G)$. Hence, $|D|=6$ implies $D\cong S_3$.
It remains to deal with the case where $P$ is not elementary abelian. In this case, a result of Laffey--MacHale~\cite[Theorem~2]{LaffeyMacHale} shows that $P=P_1\rtimes Q$ where $P_1$ is an elementary abelian $p$-group and $Q$ has order $q\in\pi\setminus\{p\}$. Moreover, $|P_1|\ge p^{q-1}$. In particular, $p>2$ since $|D|_2\le 2$. Again $B$ contains a $p$-block $B_p$ with normal defect group $P_1$ and $k(B_p)=3$. As before, we obtain $|P_1|=3$ and $q=2$. This leads to $D\cong S_3$.

Conversely, let $D\in\{C_3,S_3\}$. By the first part of the proof, $k(B)\ge 3$.
Let $N:=\pcore_{\pi'}(G)$. Using the Fong--Reynolds Theorem for $\pi$-blocks again, we may assume that $B=\Irr(G|\lambda)$ where $\lambda\in\Irr(N)$ is $G$-invariant and $D$ is a Hall $\pi$-subgroup of $G$. By a result of Gallagher (see \cite[Theorem~5.16]{Navarro2}), we have $k(B)\le k(G/N)$. Moreover, $\pcore_{\pi}(G/N)\le DN/N$ and 
\[\C_{G/N}(\pcore_{\pi}(G/N))\le\pcore_{\pi}(G/N)\] 
by the Hall--Higman Lemma mentioned above. It is easy to see that this implies $G/N\le S_3$. Hence, $k(B)\le 3$ and we are done. 
\end{proof}

\section*{Acknowledgment}
The author is supported by the German Research Foundation (\mbox{SA 2864/1-1} and \mbox{SA 2864/3-1}).

\end{document}